\documentclass[11pt, final]{article}
\usepackage{a4}
\usepackage{amsmath}%
\usepackage{amstext}%
\usepackage{amssymb}%
\usepackage{showkeys}%
\usepackage{epsfig}%
\usepackage{cite}
\usepackage{tikz}
\usepackage{subfig}
\usepackage{caption}
\usepackage{float}
\usepackage{multirow}
\usepackage{booktabs}

\usepackage{listings}
\usepackage{geometry}
\newtheorem{theorem}{Theorem}

\newtheorem{algorithm}[theorem]{Algorithm}

\newtheorem{lemma}[theorem]{Lemma}

\newtheorem{pb}[theorem]{Problem}
\newenvironment{problem}{\pb\rm}{\endpb}

\newtheorem{rem}[theorem]{Remark}
\newenvironment{remark}{\rem\rm}{\endrem}

\newcounter{unnumber}

\newtheorem{prf}[unnumber]{Proof.}
\newenvironment{proof}{\prf\rm}{\hfill{$\blacksquare$}\endprf}
\newcommand{\R}{\mathbb{R}}%
\newcommand{\N}{\mathbb{N}}%
\newcommand{\ol}{\overline}%

\DeclareMathOperator*\inte{int}%
\DeclareMathOperator*\sqri{sqri}%
\DeclareMathOperator*\ri{ri}%
\DeclareMathOperator*\dom{dom}%
\DeclareMathOperator*\B{\overline{\R}}%
\DeclareMathOperator*\gr{Gr}%
\DeclareMathOperator*\ran{ran}%
\DeclareMathOperator*\id{Id}%
\DeclareMathOperator*\prox{prox}%
\DeclareMathOperator*\argmin{argmin}

\DeclareMathOperator*\zer{zer}

\title{Penalty schemes with inertial effects for monotone inclusion problems}

\author{Radu Ioan Bo\c{t} \thanks{University of Vienna, Faculty of Mathematics, Oskar-Morgenstern-Platz 1, A-1090 Vienna, Austria, 
email: radu.bot@univie.ac.at} \and 
Ern\"{o} Robert Csetnek \thanks {University of Vienna, Faculty of Mathematics, Oskar-Morgenstern-Platz 1, A-1090 Vienna, Austria, 
email: ernoe.robert.csetnek@univie.ac.at. Research supported by FWF (Austrian Science Fund), Lise Meitner Programme, project M 1682-N25.}}
\begin{document}
\maketitle

\noindent \textbf{Abstract.} We introduce a penalty term-based splitting algorithm with inertial effects designed for solving monotone inclusion problems involving the sum
of maximally monotone operators and the convex normal cone to the (nonempty) set of zeros of a monotone and Lipschitz continuous operator. We show weak ergodic 
convergence of the generated sequence of iterates to a solution of the monotone inclusion problem, provided a condition expressed 
via the Fitzpatrick function of the operator describing the underlying set of the normal cone is verified. Under strong monotonicity assumptions
we can even show strong nonergodic convergence of the iterates. This approach constitutes the starting point for investigating from a similar perspective
monotone inclusion problems involving linear compositions of parallel-sum operators and, further, for the minimization of a
complexly structured convex objective function subject to the set of minima of another convex and differentiable  function.\vspace{1ex}

\vspace{1ex}


\section{Introduction and preliminaries}\label{sec1}

The article \cite{att-cza-10} has represented the starting point of the investigations of several authors (see \cite{att-cza-10, att-cza-peyp-c, att-cza-peyp-p, noun-peyp, peyp-12, 
b-c-penalty, bc-viet, bb, b-c-dyn-pen-cont}) 
related to the solving of variational inequalities expressed as monotone inclusion of the form 
\begin{equation}\label{att-cza-peyp-p}
0\in Ax+N_M(x),
\end{equation}
where ${\cal H}$ is a real Hilbert space, $A:{\cal H}\rightrightarrows{\cal H}$ is a maximally monotone operator, $M:=\argmin \Psi$
is the set of global minima of the convex and differentiable function
$\Psi : {\cal H} \rightarrow \R$ fulfilling $\min \Psi=0$ and $N_M:{\cal H}\rightrightarrows{\cal H}$
denotes the normal cone to the set $M \subseteq {\cal H}$. One motivation for studying such monotone inclusions certainly comes from the fact that, 
when $A = \partial \Phi$ is the convex subdifferential of a proper, convex and lower semicontinuous function $\Phi : {\cal H} \rightarrow \overline \R$.  
this opens the gates towards the solving of convex minimization problems of type
\begin{equation}
\min_{x\in{\cal H}}\{\Phi(x):x\in\argmin\Psi\}.
\end{equation}
A fruitful approach proposed in the above-mentioned literature when numerically solving \eqref{att-cza-peyp-p} assumes the penalization of the function $\Psi$ and the performing
at each iteration of a gradient step with respect to it in combination with a proximal step with respect to $A$. In the asymptotic analysis of these schemes, a fundamental role 
is played by the condition
\begin{equation}\label{rel}
\mbox{ for every }p\in\ran N_M, \sum_{n \geq 1} \lambda_n\beta_n\left[\Psi^*\left(\frac{p}{\beta_n}\right)-\sigma_M\left(\frac{p}{\beta_n}\right)\right]<+\infty, 
\end{equation}
which is basically the discrete counterpart of a condition given in the continuous setting for nonautonomous differential inclusions in \cite{att-cza-10}.
Here, $\Psi^* : {\cal H} \rightarrow \overline \R$ denotes
the Fenchel conjugate function of $\Psi$, $\ran N_M$ the range of the
normal cone operator $N_M:{\cal H}\rightrightarrows{\cal H}$, $\sigma_M$ the support function of $M$ and 
$(\lambda_n)_{n \geq 1}$ and $(\beta_n)_{n \geq 1}$ are positive real sequences representing step sizes and penalty parameters, respectively. 
For conditions guaranteeing \eqref{rel} we refer the reader to \cite{att-cza-10, att-cza-peyp-c, att-cza-peyp-p, noun-peyp, peyp-12}.

In \cite{b-c-penalty,bc-viet} we investigated from a similar perspective the more general inclusion problem
\begin{equation}\label{att-cza-peyp-p-gen}
0\in Ax+Dx+N_M(x),
\end{equation}
where $A :{\cal H}\rightrightarrows{\cal H}$ is a maximally monotone operator, $D :{\cal H}\rightarrow {\cal H}$ a
(single-valued) monotone and Lipschitz continuous operator and $M \subseteq {\cal H}$ denotes the nonempty 
set of zeros of another monotone and Lipschitz continuous operator $B :{\cal H}\rightarrow {\cal H}$.  We formulated a forward-backward-forward algorithm of penalty type for solving 
\eqref{att-cza-peyp-p-gen} and proved weak ergodic convergence for the sequence of generated iterates, provided that a condition formulated in the spirit of \eqref{rel}, however, formulated  by using the Fitzpatrick function associated to $B$, is valid.

In this paper, our aim is to endow the forward-backward-forward algorithm of penalty type for solving \eqref{att-cza-peyp-p-gen} from \cite{b-c-penalty,bc-viet} with inertial effects. Iterative schemes with inertial effects have their roots in the implicit discretization of a differential system of second-order in time  (see \cite{alvarez-attouch2001}). One of the main features of the inertial splitting algorithm is that the new iterate is defined by making use of the previous two iterates. Taking into account
the "prehistory`` of the process can lead to an acceleration of the convergence of the iterates, as
it has been for instance pointed out by Polyak (see \cite{polyak}) in the context of minimizing a differentiable function. As emphasized by 
Bertsekas in \cite{bertsekas} (see also \cite{op} and \cite{bcl}),   
one of the aspects which makes algorithms with inertial (sometimes also called memory) effects
useful is their ability to detect optimal solutions of minimization problems which cannot be found by 
their noninertial variants. Since their introduction one can notice an increasing interest in inertial 
algorithms, which is exemplified by the following references 
\cite{alvarez2000, alvarez2004, alvarez-attouch2001, att-peyp-red, b-c-inertial, b-c-inertial-admm, b-c-inertialhybrid, b-c-inertial-nonc-ts, 
bch, bcl, cabot-frankel2011, chen-chan-ma-yang, chen-ma-yang, mainge2008, mainge-moudafi2008, moudafi-oliny2003}.

We show weak ergodic convergence of the sequence generated by the proposed iterative scheme to a solution of the monotone inclusion problem \eqref{att-cza-peyp-p-gen} by using Fej\'{e}r monotonicity techniques. The correspondent of condition \eqref{rel} in the context of monotone inclusion problems of type \eqref{att-cza-peyp-p-gen} will play again a decisive role. When the operator $A$ is assumed to be strongly monotone, the iterates are proved to converge strongly to the unique solution of 
\eqref{att-cza-peyp-p-gen}. By using a product space approach, we are also able to enhance these investigations to monotone inclusion problems involving linear compositions of parallel-sum operators. 
This further allows to formulate a penalty scheme with inertial effects for the minimization of a
complexly structured convex objective function subject to the set of minima of another convex and differentiable function.

Next we present some notations which are used throughout the paper
(see \cite{bo-van, b-hab, bauschke-book, EkTem, simons, Zal-carte}). Let ${\cal H}$ be a real Hilbert space with \textit{inner product} $\langle\cdot,\cdot\rangle$
and associated \textit{norm} $\|\cdot\|=\sqrt{\langle \cdot,\cdot\rangle}$. When ${\cal G}$ is another Hilbert space and $L:{\cal H} \rightarrow {\cal G}$ a linear
continuous operator, then the \textit{norm} of $L$ is defined as $\|L\| = \sup\{\|Lx\|: x \in {\cal H}, \|x\| \leq 1\}$, while
$L^* : {\cal G} \rightarrow {\cal H}$, defined by $\langle L^*y,x\rangle = \langle y,Lx \rangle$ for all $(x,y) \in {\cal H} \times {\cal G}$,
denotes the \textit{adjoint operator} of $L$.

For a function $f:{\cal H}\rightarrow\overline{\R}$ we denote by $\dom f=\{x\in {\cal H}:f(x)<+\infty\}$ its \textit{effective domain} and say that
$f$ is \textit{proper}, if $\dom f\neq\emptyset$ and $f(x)\neq-\infty$ for all $x\in {\cal H}$.  Let $f^*:{\cal H} \rightarrow \overline \R$,
$f^*(u)=\sup_{x\in {\cal H}}\{\langle u,x\rangle-f(x)\}$ for all $u\in {\cal H}$, be the \textit{conjugate function} of $f$.
We denote by $\Gamma({\cal H})$ the family of proper, convex and lower semi-continuous extended real-valued functions defined on ${\cal H}$.
The \textit{subdifferential} of $f$ at $x\in {\cal H}$, with $f(x)\in\R$, is the set
$\partial f(x):=\{v\in {\cal H}:f(y)\geq f(x)+\langle v,y-x\rangle \ \forall y\in {\cal H}\}$. We take by convention
$\partial f(x):=\emptyset$, if $f(x)\in\{\pm\infty\}$. We also denote by $\min f := \inf_{x \in {\cal H}} f(x)$ and by
$\argmin f :=\{x \in {\cal H}: f(x) = \min f \}$. For $f,g:{\cal H}\rightarrow \overline{\R}$ two proper functions, we consider
their \textit{infimal convolution}, which is the function $f\Box g:{\cal H}\rightarrow\B$, 
$(f\Box g)(x)=\inf_{y\in {\cal H}}\{f(y)+g(x-y)\}$.

Let $M\subseteq {\cal H}$ be a nonempty convex set. The \textit{indicator function} of $M$, $\delta_M:{\cal H}\rightarrow \overline{\R}$,
is the function which takes the value $0$ on $M$ and $+\infty$ otherwise. The subdifferential of the indicator function is the
\textit{normal cone} of $M$, that is $N_M(x)=\{u\in {\cal H}:\langle u,y-x\rangle\leq 0 \ \forall y\in M\}$, if $x\in M$ and
$N_M(x)=\emptyset$ for $x\notin M$. Notice that for $x\in M$, $u\in N_M(x)$ if and only if $\sigma_M(u)=\langle u,x\rangle$,
where $\sigma_M$ is the support function of $M$, defined by $\sigma_M(u)=\sup_{y\in M}\langle y,u\rangle$. Further,
we denote by $$\sqri M:=\{x\in M:\cup_{\lambda>0}\lambda(M-x) \ \mbox{is a closed linear subspace of} \ {\cal H}\}$$
the \textit{strong quasi-relative interior} of $M$. We always have $\inte M\subseteq\sqri M$ (in general this inclusion may be strict).
If ${\cal H}$ is finite-dimensional, then $\sqri M$ coincides with $\ri M$, the relative interior of $M$, which is the interior of $M$ with
respect to its affine hull.

For an arbitrary set-valued operator $A:{\cal H}\rightrightarrows {\cal H}$ we denote by $\gr A=\{(x,u)\in {\cal H}\times {\cal H}:u\in Ax\}$
its \emph{graph}, by $\dom A=\{x \in {\cal H} : Ax \neq \emptyset\}$ its \emph{domain}, by
$\ran A=\{u\in {\cal H}: \exists x\in {\cal H} \mbox{ s.t. }u\in Ax\}$ its \emph{range} and by $A^{-1}:{\cal H}\rightrightarrows {\cal H}$
its \emph{inverse operator}, defined by $(u,x)\in\gr A^{-1}$ if and only if $(x,u)\in\gr A$. The \textit{parallel sum} of two set-valued operators $A_1, A_2: {\cal H}\rightrightarrows {\cal H}$ is defined as
$$A_1 \Box A_2 : {\cal H}\rightrightarrows {\cal H},  A_1 \Box A_2  = \left(A_1^{-1} + A_2^{-1}\right)^{-1}.$$
We use also the notation
$\zer A=\{x\in {\cal H}: 0\in Ax\}$ for the \textit{set of zeros} of the operator $A$. We say that $A$ is \emph{monotone} if
$\langle x-y,u-v\rangle\geq 0$ for all $(x,u),(y,v)\in\gr A$. A monotone operator $A$ is said to be \emph{maximally monotone}, if
there exists no proper monotone extension of the graph of $A$ on ${\cal H}\times {\cal H}$. Let us mention that in case $A$ is
maximally monotone, $\zer A$ is a convex and closed set \cite[Proposition 23.39]{bauschke-book}. We refer to \cite[Section 23.4]{bauschke-book} for
conditions ensuring that $\zer A$ is nonempty.

If $A$ is maximally monotone, then one has the following characterization for the set of its zeros
\begin{equation}\label{charact-zeros-max}
z\in\zer M \mbox{ if and only if }\langle u-z,w\rangle\geq 0\mbox{ for all }(u,w)\in \gr M.
\end{equation}

The operator $A$ is said to be \emph{$\gamma$-strongly monotone} with $\gamma>0$, if $\langle x-y,u-v\rangle\geq \gamma\|x-y\|^2$ for
all $(x,u),(y,v)\in\gr A$. Notice that if $A$ is maximally monotone and strongly monotone, then $\zer A$ is a singleton, thus nonempty
(see \cite[Corollary 23.37]{bauschke-book}). Let $\gamma>0$ be arbitrary. A single-valued operator $A:{\cal H}\rightarrow {\cal H}$ is said 
to be \textit{$\gamma$-Lipschitz continuous},
if $\|Ax-Ay\|\leq \gamma\|x-y\|$ for all $(x,y)\in {\cal H}\times {\cal H}$.

The \emph{resolvent} of $A$, $J_A:{\cal H} \rightrightarrows {\cal H}$, is defined by $J_A=(\id+A)^{-1}$, where
$\id :{\cal H} \rightarrow {\cal H}, \id(x) = x$ for all $x \in {\cal H}$, is the \textit{identity operator} on ${\cal H}$.
Moreover, if $A$ is maximally monotone, then $J_A:{\cal H} \rightarrow {\cal H}$ is single-valued and maximally monotone
(cf. \cite[Proposition 23.7 and Corollary 23.10]{bauschke-book}). For an arbitrary $\gamma>0$ we have (see \cite[Proposition 23.18]{bauschke-book})
\begin{equation}\label{j-inv-op}
J_{\gamma A}+\gamma J_{\gamma^{-1}A^{-1}}\circ \gamma^{-1}\id=\id.
\end{equation}

When $f\in\Gamma({\cal H})$ and $\gamma > 0$, for every $x \in {\cal H}$ we denote by $\prox_{\gamma f}(x)$ the \textit{proximal point}
of parameter $\gamma$ of $f$ at $x$, which is the unique optimal solution of the optimization problem
\begin{equation}\label{prox-def}\inf_{y\in {\cal H}}\left \{f(y)+\frac{1}{2\gamma}\|y-x\|^2\right\}.
\end{equation}
Notice that $J_{\gamma\partial f}=(\id+\gamma\partial f)^{-1}=\prox_{\gamma f}$, thus
$\prox_{\gamma f} :{\cal H} \rightarrow {\cal H}$ is a single-valued operator fulfilling the extended \textit{Moreau's decomposition formula}
\begin{equation}\label{prox-f-star}
\prox\nolimits_{\gamma f}+\gamma\prox\nolimits_{(1/\gamma)f^*}\circ\gamma^{-1}\id=\id.
\end{equation}
Let us also recall that the function $f:{\cal H} \rightarrow \overline \R$ is said to be \textit{$\gamma$-strongly convex}
for $\gamma >0$, if $f-\frac{\gamma}{2}\|\cdot\|^2$ is a convex function. Let us mention that this property implies that $\partial f$ is $\gamma$-strongly monotone (see \cite[Example 22.3]{bauschke-book}).

The \emph{Fitzpatrick function} associated to a monotone operator $A$ is defined as
$$\varphi_A:{\cal H}\times {\cal H}\rightarrow \B, \ \varphi_A(x,u)=\sup_{(y,v)\in\gr A}\{\langle x,v\rangle+\langle y,u\rangle-\langle y,v\rangle\},$$
and it is a convex and lower semicontinuous function.  Introduced by Fitzpatrick in \cite{fitz},
this functione opened the gate towards the employment of convex analysis specific tools when investigating the maximality of monotone operators
(see \cite{bauschke-book, bausch-m-s, b-hab, BCW-set-val, bo-van, borw-06, bu-sv-02, simons} and the references therein) and it will play an important role throughout the paper. In case $A$ is
maximally monotone, $\varphi_A$ is proper and it fulfills
$$\varphi_A(x,u)\geq \langle x,u\rangle \ \forall (x,u)\in {\cal H}\times {\cal H},$$
with equality if and only if $(x,u)\in\gr A$. Notice that if $f\in\Gamma(\cal H)$,
then $\partial f$ is a maximally monotone operator (cf. \cite{rock}) and it holds $(\partial f)^{-1} = \partial f^*$. Furthermore, the following
inequality is true (see \cite{bausch-m-s})
\begin{equation}\label{fitzp-subdiff-ineq}
 \varphi_{\partial f}(x,u)\leq f(x) +f^*(u) \ \forall (x,u)\in {\cal H}\times {\cal H}.
\end{equation}
We refer the reader to \cite{bausch-m-s} for formulae of the Fitzpatrick function computed for particular classes of monotone
operators.

We close the section by presenting some convergence results that will be used several times in the paper. 
Let $(x_n)_{n \geq 1}$ be a sequence in ${\cal H}$ and $(\lambda_k)_{k \geq 1}$ a sequence of positive numbers such that 
$\sum_{k \geq 1} \lambda_k=+\infty$. Let $(z_n)_{n \geq 1} $ be the sequence of weighted averages defined as (see \cite{att-cza-peyp-c})
\begin{equation}\label{average}
z_n=\frac{1}{\tau_n}\sum_{k=1}^n\lambda_k x_k, \mbox{ where }\tau_n= \sum_{k=1}^n\lambda_k \ \forall n \geq 1.
\end{equation}

\begin{lemma}\label{opial-passty} (Opial-Passty) Let $F$ be a nonempty subset of ${\cal H}$ and assume 
that $\lim_{n\rightarrow\infty}\|x_n-x\|$ exists for every $x\in F$. If every sequential weak cluster 
point of $(x_n)_{n \geq 1}$ (respectively $(z_n)_{n \geq 1}$) lies in $F$, then $(x_n)_{n \geq 1}$ 
(respectively $(z_n)_{n \geq 1}$) converges weakly to an element in $F$ as $n\rightarrow+\infty$.
\end{lemma}

\begin{lemma}\label{ext-fejer1} (see \cite{alvarez-attouch2001, alvarez2000, alvarez2004}) Let $(\varphi_n)_{n\geq 0}, (\delta_n)_{n\geq 1}$ 
and $(\alpha_n)_{n\geq 1}$ be sequences in
$[0,+\infty)$ such that $\varphi_{n+1}\leq\varphi_n+\alpha_n(\varphi_n-\varphi_{n-1})+\delta_n$
for all $n \geq 1$, $\sum_{n\geq 1}\delta_n< + \infty$ and there exists a real number $\alpha$ with
$0\leq\alpha_n\leq\alpha<1$ for all $n\geq 1$. Then the following statements are true: 
\begin{itemize}\item[(i)] $\sum_{n \geq 1}[\varphi_n-\varphi_{n-1}]_+< + \infty$, where
$[t]_+=\max\{t,0\}$; \item[(ii)] there exists $\varphi^*\in[0,+\infty)$ such that $\lim_{n\rightarrow+\infty}\varphi_n=\varphi^*$.\end{itemize}
\end{lemma}

A direct consequence of Lemma \ref{ext-fejer1} is the following result.

\begin{lemma}\label{ext-fejer2} Let $(\varphi_n)_{n\geq 0}, (\delta_n)_{n\geq 1}, (\alpha_n)_{n\geq 1}$ and $(\beta_n)_{n\geq 1}$ 
be sequences in $[0,+\infty)$ such that $\varphi_{n+1}\leq-\beta_n+\varphi_n+\alpha_n(\varphi_n-\varphi_{n-1})+\delta_n$
for all $n \geq 1$, $\sum_{n\geq 1}\delta_n<+\infty$ and there exists a real number $\alpha$ with
$0\leq\alpha_n\leq\alpha<1$ for all $n\geq 1$. Then the following hold: \begin{itemize}\item[(i)] $\sum_{n \geq 1}[\varphi_n-\varphi_{n-1}]_+<+\infty$, where
$[t]_+=\max\{t,0\}$; \item[(ii)] there exists $\varphi^*\in[0,+\infty)$ such that $\lim_{n\rightarrow + \infty}\varphi_n=\varphi^*$; 
\item[(iii)] $\sum_{n\geq 1}\beta_n<+\infty$.\end{itemize}
\end{lemma}

\section{A forward-backward-forward penalty algorithm with inertial effects}\label{sec2}

Throughout this section we are concerned with the solving of the following monotone inclusion problem.

\begin{problem}\label{pr-Lip-single-val}
Let ${\cal H}$ be a real Hilbert space, $A:{\cal H}\rightrightarrows {\cal H}$ a maximally monotone operator, 
$D:{\cal H}\rightarrow{\cal H}$ a monotone and $\eta^{-1}$-Lipschitz continuous operator with $\eta > 0$, 
$B:{\cal H}\rightarrow{\cal H}$ a monotone and $\mu^{-1}$-Lipschitz continuous operator with $\mu>0$ and 
assume that $M=\zer B \neq\emptyset$. The monotone inclusion problem to solve is
$$0\in Ax+Dx+N_M(x).$$
\end{problem}

We propose the following iterative scheme for solving Problem \ref{pr-Lip-single-val}.

\begin{algorithm}\label{alg-fbf}$ $

\noindent\begin{tabular}{rl}
\verb"Initialization": & \verb"Choose" $x_0,x_1\in{\cal H}$\\
\verb"For" $n \geq 1 $ \verb"set": & $p_n=J_{\lambda_n A}(x_n-\lambda_nDx_n-\lambda_n\beta_nBx_n+\alpha_n(x_n-x_{n-1}))$\\
                                &  $x_{n+1}=\lambda_n\beta_n(Bx_n-Bp_n)+\lambda_n(Dx_n-Dp_n)+p_n$,
\end{tabular}
\end{algorithm}
where $(\lambda_n)_{n \geq 1}$, $(\beta_n)_{n \geq 1}$ and $(\alpha_n)_{n \geq 1}$ are sequences of positive real numbers that represent the step sizes, the penalty parameters and the control parameters of the inertial effects, respectively.

\begin{remark} When $\alpha_n=0$ for any $n\geq 1$, the above numerical scheme becomes Algorithm 3 
in \cite{b-c-penalty}. On the other hand, assume that $Bx=0$ for all $x\in {\cal H}$ (having as consequence $M=\cal H$ and $N_M(x)=\{0\}$). In this case, Algorithm \ref{alg-fbf}  turns out to be the inertial splitting method proposed and analyzed in \cite{b-c-inertial} for solving the monotone inclusion problem \begin{equation}\label{a-d}0\in Ax+Dx.\end{equation} 
If we combine these two cases, namely by assuming that $\alpha_n=0$ for any $n\geq 1$ and $Bx=0$ for all $x\in {\cal H}$, then Algorithm \ref{alg-fbf} is nothing else than Tseng's iterative scheme for solving \eqref{a-d} (see also \cite{br-combettes} for an error tolerant version 
of this method). 
\end{remark}

The following technical statement will be useful in the convergence analysis of Algorithm \ref{alg-fbf}. 

\begin{lemma}\label{fbf-ineq1}  Let $(x_n)_{n \geq 0}$ and $(p_n)_{n \geq 1}$ be the sequences generated by Algorithm \ref{alg-fbf} and 
let $(u,w) \in \gr(A+D+N_M)$ be such that $w=v+p+Du$, where $v\in Au$ and $p\in N_M(u)$. 
Then the following inequality holds for any $n\geq 1$:
\begin{align}\label{lem19}
\|x_{n+1}-u\|^2-\|x_n-u\|^2 \leq & \ \alpha_n(\|x_n-u\|^2-\|x_{n-1}-u\|^2)+ 2\alpha_n\|x_n-x_{n-1}\|^2\nonumber\\& \ \--\left[1-\left(\frac{\lambda_n\beta_n}{\mu} + \frac{\lambda_n}{\eta} \right)^2-\alpha_n\right]\|x_n-p_n\|^2 \nonumber\\
& \  + \ 2\lambda_n\beta_n\left[\sup_{u\in M}\varphi_B\left(u,\frac{p}{\beta_n}\right)-\sigma_M\left(\frac{p}{\beta_n}\right)\right]+2\lambda_n\langle u-p_n,w\rangle.
\end{align}
\end{lemma}
\begin{proof} It follows from the definition of the resolvent operator that $\frac{1}{\lambda_n}(x_n-p_n)-\beta_nBx_n-Dx_n+\frac{\alpha_n}{\lambda_n}(x_n-x_{n-1})\in Ap_n$ 
for any $n \geq 1$ and, since $v\in Au$, the monotonicity of $A$ guarantees
$$\langle p_n-u,x_n-p_n-\lambda_n(\beta_nBx_n+Dx_n+v)+\alpha_n(x_n-x_{n-1})\rangle\geq 0 \ \forall n \geq 1,$$
thus
$$\langle u-p_n,x_n-p_n\rangle\leq \langle u-p_n,\lambda_n\beta_nBx_n+\lambda_nDx_n+\lambda_nv-\alpha_n(x_n-x_{n-1})\rangle \ \forall n \geq 1.$$

In the following we take into account the definition of $x_{n+1}$ given in the algorithm and obtain
\begin{align}\label{ineq}
 \langle u-p_n,x_n-p_n\rangle\leq & \ \langle u-p_n,x_{n+1}-p_n+\lambda_n\beta_nBp_n+\lambda_nDp_n+\lambda_nv-\alpha_n(x_n-x_{n-1})\rangle\nonumber\\
= & \ \langle u-p_n,x_{n+1}-p_n\rangle+\lambda_n\beta_n\langle u-p_n,Bp_n\rangle+\lambda_n\langle u-p_n,Dp_n\rangle\nonumber\\
&+\lambda_n\langle u-p_n,v\rangle +\alpha_n\langle p_n-u,x_n-x_{n-1}\rangle\ \forall n \geq 1.
\end{align}

Notice that for any $n\geq 1$ $$\langle u-p_n,x_n-p_n\rangle=\frac{1}{2}\|u-p_n\|^2-\frac{1}{2}\|x_n-u\|^2+\frac{1}{2}\|x_n-p_n\|^2,$$
$$\langle u-p_n,x_{n+1}-p_n\rangle=\frac{1}{2}\|u-p_n\|^2-\frac{1}{2}\|x_{n+1}-u\|^2+\frac{1}{2}\|x_{n+1}-p_n\|^2$$
and \begin{align*}\langle p_n-u,x_n-x_{n-1}\rangle&=\langle x_n-u,x_n-x_{n-1}\rangle+\langle p_n-x_n,x_n-x_{n-1}\rangle\\
&=\frac{\|x_n-x_{n-1}\|^2}{2}+\frac{\|x_n-u\|^2}{2}-\frac{\|x_{n-1}-u\|^2}{2}\\
& \ \ \ +\frac{\|p_n-x_{n-1}\|^2}{2}-\frac{\|x_n-x_{n-1}\|^2}{2}-\frac{\|x_n-p_n\|^2}{2}.
\end{align*}

By making use of these equalities, from \eqref{ineq} we obtain that for any $n\geq 1$
\begin{align*}
& \ \frac{1}{2}\|u-p_n\|^2-\frac{1}{2}\|x_n-u\|^2+\frac{1}{2}\|x_n-p_n\|^2\\
\leq & \ \frac{1}{2}\|u-p_n\|^2-\frac{1}{2}\|x_{n+1}-u\|^2+\frac{1}{2}\|x_{n+1}-p_n\|^2 +\\
&  \ \lambda_n\beta_n\langle u-p_n,Bp_n\rangle+\lambda_n\langle u-p_n,Dp_n\rangle+\lambda_n\langle u-p_n,v\rangle + \\
& \ \alpha_n\left[\frac{\|x_n-u\|^2}{2}-\frac{\|x_{n-1}-u\|^2}{2}+\frac{\|p_n-x_{n-1}\|^2}{2}-\frac{\|x_n-p_n\|^2}{2}\right].
\end{align*}

Further, by using the inequality $$\|p_n-x_{n-1}\|^2\leq 2\|x_n-p_n\|^2+2\|x_n-x_{n-1}\|^2,$$ the 
relation $v=w-p-Du$ and the definition of the Fitzpatrick function we derive for any $n\geq 1$
\begin{align*}
& \ \|x_{n+1}-u\|^2-\|x_n-u\|^2\\
\leq & \ \|x_{n+1}-p_n\|^2-\|x_n-p_n\|^2+2\lambda_n\beta_n\left(\langle u,Bp_n\rangle+\left \langle p_n,\frac{p}{\beta_n} \right \rangle-\langle p_n,Bp_n\rangle
-\left \langle u,\frac{p}{\beta_n} \right \rangle\right)+\\
& 2\lambda_n\langle u-p_n,Dp_n-Du\rangle+2\lambda_n\langle u-p_n,w\rangle+\\
& \alpha_n\left[\|x_n-u\|^2-\|x_{n-1}-u\|^2+\|x_n-p_n\|^2+2\|x_n-x_{n-1}\|^2\right]\\
\leq & \ \|x_{n+1}-p_n\|^2-\|x_n-p_n\|^2 +2\lambda_n\beta_n\left[\sup_{u\in M} \varphi_B\left(u,\frac{p}{\beta_n}\right)-\sigma_M\left(\frac{p}{\beta_n}\right)\right]+\\
&2\lambda_n\langle u-p_n,Dp_n-Du\rangle+2\lambda_n\langle u-p_n,w\rangle+\\
&\alpha_n\left[\|x_n-u\|^2-\|x_{n-1}-u\|^2+\|x_n-p_n\|^2+2\|x_n-x_{n-1}\|^2\right].
\end{align*}

Since $D$ is monotone, we have $\langle u-p_n,Dp_n-Du\rangle\leq 0$ for any $n\geq 1$ and the conclusion follows by noticing that the 
Lipschitz continuity of $B$ and $D$ yields 
$$\|x_{n+1}-p_n\|\leq \frac{\lambda_n\beta_n}{\mu} \|x_n-p_n\|+\frac{\lambda_n}{\eta}\|x_n-p_n\| = 
\left (\frac{\lambda_n\beta_n}{\mu} + \frac{\lambda_n}{\eta} \right)\|x_n-p_n\| \ \forall n \geq 1.$$
\end{proof}

We will prove the convergence of Algorithm \ref{alg-fbf} under the following hypotheses:
$$(H_{fitz})\left\{
\begin{array}{lll}
(i) \ A+N_M \mbox{ is maximally monotone and }\zer(A+D+N_M)\neq\emptyset;\\
(ii) \ \mbox{ For every }p\in\ran N_M, \sum_{n \geq 1} \lambda_n\beta_n\left[\sup\limits_{u\in M}\varphi_B\left(u,\frac{p}{\beta_n}\right)-\sigma_M\left(\frac{p}{\beta_n}\right)\right]<+\infty;\\
(iii) \ (\lambda_n)_{n \geq 1} \in\ell^2\setminus\ell^1.\end{array}\right.$$

\begin{remark}\label{cond} The first part of the statement in (i) is verified if one of the Rockafellar conditions 
$M\cap\inte \dom A \neq\emptyset$ or $\dom A\cap \inte M\neq\emptyset$ is fulfilled (see \cite{rock-cond}). We refer the reader to \cite{bauschke-book, b-hab, BCW-set-val, bo-van, borw-06, simons} 
for further conditions which guarantee the maximality
of the sum of maximally monotone operators. Further, we refer to \cite[Subsection 23.4]{bauschke-book} for conditions ensuring 
that the set of zeros of a maximally monotone operator is nonempty. The condition (ii) above has been introduced 
for the first time in \cite{b-c-penalty}. According to \cite[Remark 4]{b-c-penalty}, the hypothesis (ii) is a 
generalization of the condition considered in \cite{att-cza-peyp-c} (see also $(H_{fitz}^{opt})$ and Remark \ref{cond-subdiff} 
in Section 4 for conditions guaranteeing (ii)).
\end{remark}

\begin{remark} (see also \cite{b-c-penalty}) Since $D$ is maximally monotone (see \cite[Example 20.28]{bauschke-book}) and $\dom D={\cal H}$, the hypothesis (i) 
above guarantees that $A+D+N_M$ is maximally monotone, too (see \cite[Corollary 24.4]{bauschke-book}). Moreover, for each 
$p\in\ran N_M$ we have
$$\sup\limits_{u\in M}\varphi_B\left(u,\frac{p}{\beta_n}\right)-\sigma_M\left(\frac{p}{\beta_n}\right)\geq 0 \ \forall n\geq 1.$$
Indeed, if $p\in\ran N_M$, then there exists $\ol u\in M$ such that $p\in N_M(\ol u)$. This implies that
$$\sup\limits_{u\in M}\varphi_B\left(u,\frac{p}{\beta_n}\right)-\sigma_M\left(\frac{p}{\beta_n}\right)
\geq \left\langle \ol u,\frac{p}{\beta_n}\right\rangle-\sigma_M\left(\frac{p}{\beta_n}\right)=0 \ \forall n \geq 1.$$\end{remark}

Let us state now the convergence properties of the sequences generated by Algorithm \ref{alg-fbf}.

\begin{theorem}\label{fbf-conv} Let $(x_n)_{n \geq 0}$ and $(p_n)_{n \geq 1}$ be the sequences generated 
by Algorithm \ref{alg-fbf} and $(z_n)_{n \geq 1}$ the sequence defined in \eqref{average}. 
Assume that $(H_{fitz})$ is fulfilled, $(\alpha_n)_{n\geq 1}$ 
is nondecreasing and there exist $n_0\geq 1$, $\alpha\geq 0$ and $\sigma>0$ such that for any $n \geq n_0$
\begin{equation}\label{a_n}0\leq\alpha_n\leq \alpha \end{equation} and 
\begin{equation}\label{a-s}5\alpha+2\sigma+(1+4\alpha+2\sigma)\left (\frac{\lambda_n\beta_n}{\mu} + \frac{\lambda_n}{\eta} \right)^2\leq 1.\end{equation}
The following statements hold: 
\begin{enumerate}
 \item [(i)] $\sum_{n\geq 0}\|x_{n+1}-x_n\|^2<+\infty$ and $\sum_{n\geq 1}\|x_n-p_n\|^2<+\infty$; 
 \item[(ii)] $(z_n)_{n \geq 1}$ converges weakly to an element in $\zer(A+D+N_M)$ as $n\rightarrow+\infty$;
 \item[(iii)] if $A$ is $\gamma$-strongly monotone with $\gamma>0$, then $(x_n)_{n\geq 0}$ and $(p_n)_{n\geq 1}$ converge 
 strongly to the unique element in $\zer(A+D+N_M)$ as $n\rightarrow+\infty$. 
\end{enumerate}
\end{theorem}

\begin{remark}\label{rem-a-s} When
$$\limsup_{n\rightarrow + \infty}\left (\frac{\lambda_n\beta_n}{\mu} + \frac{\lambda_n}{\eta} \right)<1,$$
one can select $\alpha\geq 0$, $\sigma>0$ and $n_0\geq 1$ such that \eqref{a-s} holds for any $n \geq n_0$. With given $\alpha$ and $n_0$ one can chose 
a nondecreasing sequence $(\alpha_n)_{n\geq 1}$ fulfilling \eqref{a_n} for any $n \geq n_0$, too. 
\end{remark}

\begin{proof} We start by noticing that for $(u,w) \in \gr(A+D+N_M)$ such that $w=v+p+Du$, where $v\in Au$ and $p\in N_M(u)$, from \eqref{lem19} we obtain for any $n\geq n_0$
\begin{align}\label{lem19-1}
\|x_{n+1}-u\|^2-\|x_n-u\|^2 \leq & \ \alpha_n(\|x-u\|^2-\|x_{n-1}-u\|^2)+ 2\alpha_n\|x_n-x_{n-1}\|^2\nonumber\\
& \ \--\frac{\left[1-\left(\frac{\lambda_n\beta_n}{\mu} + \frac{\lambda_n}{\eta} \right)^2-\alpha_n\right]}{\left(1+\frac{\lambda_n\beta_n}{\mu} + \frac{\lambda_n}{\eta}\right)^2}\|x_{n+1}-x_n\|^2 \nonumber\\
& \  + \ 2\lambda_n\beta_n\left[\sup_{u\in M}\varphi_B\left(u,\frac{p}{\beta_n}\right)-\sigma_M\left(\frac{p}{\beta_n}\right)\right]+2\lambda_n\langle u-p_n,w\rangle.
\end{align}
Indeed, this follows by taking into consideration that 
\begin{align*}\|x_{n+1}-x_n\|&=\|\lambda_n\beta_n(Bx_n-Bp_n)+\lambda_n(Dx_n-Dp_n)+p_n-x_n\|\\
& \leq \left(1+\frac{\lambda_n\beta_n}{\mu} + \frac{\lambda_n}{\eta}\right)\|x_n-p_n\|
\end{align*}
and,  due to \eqref{a-s},
$$1-\left(\frac{\lambda_n\beta_n}{\mu} + \frac{\lambda_n}{\eta} \right)^2-\alpha_n>0 \ \forall n\geq n_0.$$
 
With the notations $$\varphi_n:=\|x_n-u\|^2 \ \forall n\geq 0$$ and 
$$\mu_n:=\varphi_n-\alpha_n\varphi_{n-1}+2\alpha_n\|x_n-x_{n-1}\|^2 \ \forall n\geq 1,$$
we obtain from \eqref{lem19-1} and the fact that $(\alpha_n)_{n\geq 1}$ is nondecreasing the inequality 
\begin{align}\label{mu_n}\mu_{n+1}-\mu_n&\leq \varphi_{n+1}-\varphi_n-\alpha_n(\varphi_n-\varphi_{n-1})+2\alpha_{n+1}\|x_{n+1}-x_n\|^2-2\alpha_n\|x_n-x_{n-1}\|^2\nonumber\\
& \leq \left[2\alpha_{n+1}-\frac{1-\left(\frac{\lambda_n\beta_n}{\mu} + \frac{\lambda_n}{\eta} \right)^2-\alpha_n}
{\left(1+\frac{\lambda_n\beta_n}{\mu} + \frac{\lambda_n}{\eta}\right)^2}\right]\|x_{n+1}-x_n\|^2\nonumber\\
&  \ \ \  + \ 2\lambda_n\beta_n\left[\sup_{u\in M}\varphi_B\left(u,\frac{p}{\beta_n}\right)-\sigma_M\left(\frac{p}{\beta_n}\right)\right]+2\lambda_n\langle u-p_n,w\rangle \ \forall n\geq n_0.
\end{align}

Further, we claim that \begin{equation}\label{a} 2\alpha_{n+1}-\frac{1-\left(\frac{\lambda_n\beta_n}{\mu} + \frac{\lambda_n}{\eta} \right)^2-\alpha_n}
{\left(1+\frac{\lambda_n\beta_n}{\mu} + \frac{\lambda_n}{\eta}\right)^2}\leq-\sigma \ \forall n\geq n_0.\end{equation}

Indeed, this is equivalent to 
$$\left(2\alpha_{n+1}+\sigma\right)\left(1+\frac{\lambda_n\beta_n}{\mu} + \frac{\lambda_n}{\eta}\right)^2+
\alpha_n+\left(\frac{\lambda_n\beta_n}{\mu} + \frac{\lambda_n}{\eta} \right)^2\leq 1 \ \forall n\geq n_0,$$
which is true since due to \eqref{a_n}-\eqref{a-s} we have for any $n\geq n_0$
$$\left(2\alpha_{n+1}+\sigma\right)\left(1+\frac{\lambda_n\beta_n}{\mu} + \frac{\lambda_n}{\eta}\right)^2+
\alpha_n+\left(\frac{\lambda_n\beta_n}{\mu} + \frac{\lambda_n}{\eta} \right)^2\leq $$
$$\left(2\alpha+\sigma\right)\left(1+\frac{\lambda_n\beta_n}{\mu} + \frac{\lambda_n}{\eta}\right)^2
+\alpha+\left(\frac{\lambda_n\beta_n}{\mu} + \frac{\lambda_n}{\eta} \right)^2\leq$$
$$2(2\alpha+\sigma)\left(1+\left(\frac{\lambda_n\beta_n}{\mu} + \frac{\lambda_n}{\eta} \right)^2\right)
+\alpha+\left(\frac{\lambda_n\beta_n}{\mu} + \frac{\lambda_n}{\eta} \right)^2\leq 1.$$

We conclude from \eqref{mu_n}-\eqref{a} that for any $n \geq n_0$
\begin{align}\label{mu_n-1}\mu_{n+1}-\mu_n&\leq -\sigma\|x_{n+1}-x_n\|^2
+ 2\lambda_n\beta_n\left[\sup_{u\in M}\varphi_B\left(u,\frac{p}{\beta_n}\right)-\sigma_M\left(\frac{p}{\beta_n}\right)\right] + 2\lambda_n\langle u-p_n,w\rangle.
\end{align}

Next we divide the proof in three parts.

(a) First we prove that for every $u\in\zer(A+D+N_M)$ the sequence $(\|x_n-u\|)_{n \geq 0}$ is convergent and the statement (i) holds.

Take an arbitrary $u\in\zer(A+D+N_M)$. In this case one can take $w=0$ in \eqref{mu_n-1} and summing up these inequalities for 
$n=n_0,...,N$, where $N\geq n_0$, we obtain \begin{align}\label{eq20}\mu_{N+1}-\mu_{n_0}&\leq -\sigma\sum_{n=n_0}^N\|x_{n+1}-x_n\|^2
+ 2\sum_{n=n_0}^N\lambda_n\beta_n\left[\sup_{u\in M}\varphi_B\left(u,\frac{p}{\beta_n}\right)-\sigma_M\left(\frac{p}{\beta_n}\right)\right]\\
&\leq 2\sum_{n=1}^{\infty}\lambda_n\beta_n\left[\sup_{u\in M}\varphi_B\left(u,\frac{p}{\beta_n}\right)-\sigma_M\left(\frac{p}{\beta_n}\right)\right],\nonumber
\end{align}
hence $(\mu_n)_{n\geq 1}$ is bounded from above due to $(H_{fitz})$. Let $C>0$ be an upper bound of this sequence. By using \eqref{a_n} we get 
$$\varphi_n-\alpha\varphi_{n-1}\leq\mu_n\leq C \ \forall n\geq n_0,$$
from which we deduce (since $\alpha\in(0,1)$ due to \eqref{a-s})
\begin{align}\label{ineq-phi}\varphi_n&\leq \alpha^{n-n_0}\varphi_{n_0}+C\sum_{k=1}^{n-n_0}\alpha^{k-1} \leq \alpha^{n-n_0}\varphi_{n_0}+\frac{C}{1-\alpha}\ \forall n\geq n_0+1.
\end{align}

Further, by combining \eqref{eq20}, the definition of the sequence $(\mu_n)_{n\geq 1}$, \eqref{a_n} and \eqref{ineq-phi} 
we obtain for any $N\geq n_0+1$
\begin{align}\label{x-l2}\sigma\sum_{n=n_0}^N\|x_{n+1}-x_n\|^2&\leq 
\mu_{n_0}-\mu_{N+1}+2\sum_{n=n_0}^N\lambda_n\beta_n\left[\sup_{u\in M}\varphi_B\left(u,\frac{p}{\beta_n}\right)-\sigma_M\left(\frac{p}{\beta_n}\right)\right]\nonumber\\
&\leq \mu_{n_0}+\alpha\varphi_N+2\sum_{n=n_0}^N\lambda_n\beta_n\left[\sup_{u\in M}\varphi_B\left(u,\frac{p}{\beta_n}\right)-\sigma_M\left(\frac{p}{\beta_n}\right)\right]\nonumber\\
&\leq \mu_{n_0}+\alpha^{N-n_0+1}+\frac{C\alpha}{1-\alpha}+2\sum_{n=n_0}^N\lambda_n\beta_n\left[\sup_{u\in M}\varphi_B\left(u,\frac{p}{\beta_n}\right)-\sigma_M\left(\frac{p}{\beta_n}\right)\right].\nonumber\\
\end{align}
From $(H_{fitz})$ and \eqref{x-l2} we further get $\sum_{n\geq 0}\|x_{n+1}-x_n\|^2<+\infty$. 
Moreover, by taking into account the inequality \eqref{lem19} for $w=0$, \eqref{a_n}-\eqref{a-s} and $(H_{fitz})$, we conclude 
by applying Lemma \ref{ext-fejer2} that $\sum_{n\geq 1}\|x_n-p_n\|^2<+\infty$ and $(\|x_n-u\|)_{n \geq 0}$ is convergent.

(b) Next we prove that every weak cluster point of $(z_n')_{n \in \N}$, where
$$z_n':=\frac{1}{\tau_n}\sum_{k=1}^n\lambda_k p_k \ \mbox{and} \ \tau_n:= \sum_{k=1}^n\lambda_k \ \forall n \geq 1,$$
lies in $\zer(A+D+N_M)$.

Let $z$ be a sequential weak cluster point of $(z_n')_{n \geq 1}$. As we already noticed that $A+D+N_M$ is maximally monotone, in order 
to show that $z\in\zer(A+D+N_M)$ we will use the characterization given in \eqref{charact-zeros-max}.
Take $(u,w)\in\gr (A+D+N_M)$ such that $w=v+p+Du$, where $v\in Au$ and $p\in N_M(u)$. 
From \eqref{mu_n-1} we have $$\mu_{n+1}-\mu_n\leq
 2\lambda_n\beta_n\left[\sup_{u\in M}\varphi_B\left(u,\frac{p}{\beta_n}\right)-\sigma_M\left(\frac{p}{\beta_n}\right)\right]
 +2\lambda_n\langle u-p_n,w\rangle \ \forall n\geq n_0.$$

Let be $N\in\N$ with $N\geq n_0+2$. Summing up the above inequalities for $n=n_0+1,...,N$ we get
$$\mu_{N+1}-\mu_{n_0+1}\leq L+2\left\langle \sum_{n=1}^N\lambda_nu-\sum_{n=1}^N\lambda_np_n-\sum_{n=1}^{n_0}\lambda_nu+\sum_{n=1}^{n_0}\lambda_np_n,w\right\rangle,$$
where
\begin{align*}
L := \ 2\sum_{n \geq 1} \lambda_n\beta_n\left[\sup_{u\in M}\varphi_B\left(u,\frac{p}{\beta_n}\right)-\sigma_M\left(\frac{p}{\beta_n}\right)\right]
 \in \R.
\end{align*}

Dividing by $2\tau_N=2\sum_{k=1}^N\lambda_k$ we obtain
\begin{equation}\label{eqLtilde}
\frac{\mu_{N+1}-\mu_{n_0+1}}{2\tau_N}\leq \frac{\widetilde{L}}{2\tau_N}+\langle u-z_N',w\rangle,
\end{equation}
where $$\widetilde{L}:=L+2 \left \langle -\sum_{n=1}^{n_0}\lambda_nu+\sum_{n=1}^{n_0}\lambda_np_n,w \right \rangle \in \R.$$
Notice that due to (a), $(x_n)_{n\geq 0}$ is bounded, hence the sequence $(\mu_n)_{n\geq 1}$ is bounded as well. 
By passing in \eqref{eqLtilde} to limit as $N\rightarrow+\infty$ and by using that $\lim_{N\rightarrow+\infty}\tau_N=+\infty$, we get 
$$\liminf_{N\rightarrow+\infty}\langle u-z_N',w\rangle\geq 0.$$
Since $z$ is a sequential weak cluster point of $(z_n')_{n \geq 1}$, we obtain that $\langle u-z,w\rangle\geq 0$. 
Finally, as this inequality holds for arbitrary $(u,w)\in\gr (A+D+N_M)$, the desired conclusion follows.

(c) In the third part we show that every weak cluster point of $(z_n)_{n \geq 1}$ lies in $\zer(A+D+N_M)$.

For (c) it is enough to prove that $\lim_{n\rightarrow+\infty} \|z_n-z_n'\|=0$ and the statement will 
be a consequence of (b).

For any $n \geq 1$ it holds
\begin{align*}
\|z_n-z_n'\|^2= \frac{1}{\tau_n^2}\left\|\sum_{k=1}^n\lambda_k(x_k-p_k)\right\|^2 \!\!\!\leq \frac{1}{\tau_n^2}\left(\sum_{k=1}^n\lambda_k\|x_k-p_k\|\right)^2\!\!\! \leq \frac{1}{\tau_n^2}\left(\sum_{k=1}^n\lambda_k^2\right)\!\!\left(\sum_{k=1}^n\|x_k-p_k\|^2\right).
\end{align*}
Since $(\lambda_n)_{n \geq 1} \in \ell^2 \setminus \ell^1$, taking into consideration that 
$\tau_n = \sum_{k=1}^n \lambda_k \rightarrow +\infty$ as $n \rightarrow +\infty$ and $\sum_{n\geq 1}\|x_n-p_n\|^2<+\infty$, 
we obtain $\|z_n-z_n'\| \rightarrow 0$ as $n \rightarrow +\infty$. 

The statement (ii) of the theorem follows by combining the statements proved in (a) and (c) with Lemma \ref{opial-passty}. 

Finally, we prove (iii) and assume to this end that $A$ is $\gamma$-strongly monotone. Let be $u\in\zer(A+D+N_M)$ and $w=0=v+p+Du$, where $v\in Au$ and $p\in N_M(u)$. 
Following the lines of the proof of Lemma \ref{fbf-ineq1}, one obtains for any $n\geq n_0$
\begin{align}\label{lem19-2}
2\gamma\lambda_n\|p_n-u\|^2+\|x_{n+1}-u\|^2-\|x_n-u\|^2 \!\!\leq & \ \alpha_n(\|x-u\|^2-\|x_{n-1}-u\|^2) + 2\alpha_n\|x_n-x_{n-1}\|^2 \nonumber\\
& \ \--\left[1-\left(\frac{\lambda_n\beta_n}{\mu} + \frac{\lambda_n}{\eta} \right)^2-\alpha_n\right]\|x_n-p_n\|^2 \nonumber\\
& \  + \ 2\lambda_n\beta_n\left[\sup_{u\in M}\varphi_B\left(u,\frac{p}{\beta_n}\right)-\sigma_M\left(\frac{p}{\beta_n}\right)\right]\nonumber\\
\leq&  \ \alpha_n(\|x-u\|^2-\|x_{n-1}-u\|^2) + 2\alpha_n\|x_n-x_{n-1}\|^2 \nonumber\\
&  \  + \ 2\lambda_n\beta_n\left[\sup_{u\in M}\varphi_B\left(u,\frac{p}{\beta_n}\right)-\sigma_M\left(\frac{p}{\beta_n}\right)\right].
\end{align}

Invoking now Lemma \ref{ext-fejer2}, we obtain that 
$$\sum_{n \geq 1}\lambda_n\|p_n-u\|^2<+\infty.$$

Since $(\lambda_n)_{n \geq 1}$ is bounded from above and $\sum_{n \geq 1} \|x_n-p_n\|^2<+\infty$, it yields
\begin{align*}
\sum_{n=1}^\infty\lambda_n\|x_n-u\|^2 \leq 2\sum_{n=1}^\infty\lambda_n\|x_n-p_n\|^2+2\sum_{n=1}^\infty\lambda_n\|p_n-u\|^2 < +\infty.
\end{align*}
As $\sum_{n \geq 1} \lambda_n=+\infty$ and  $(\|x_n-u\|)_{n \geq 1}$ is convergent, it follows $\lim_{n\rightarrow+\infty}\|x_n-u\|=0$. 
This obviously implies $\lim_{n\rightarrow+\infty}\|p_n-u\|=0$, since $\lim_{n\rightarrow+\infty}\|x_n-p_n\|=0$. 
\end{proof}

\section{A primal-dual forward-backward-forward penalty algorithm with inertial effects}\label{sec3}

The aim of this section is to propose and investigate from the point of view of its convergence properties a forward-backward-forward penalty algorithm with inertial effects
for solving the following monotone inclusion problem involving linearly composed and parallel-sum type monotone operators.

\begin{problem}\label{pr-par-sum}
Let ${\cal H}$ be a real Hilbert space, $A:{\cal H}\rightrightarrows {\cal H}$ a maximally monotone operator and
$C:{\cal H}\rightarrow {\cal H}$ a monotone and $\nu$-Lipschitz continuous operator for $\nu>0$. Let $m$ be a strictly positive integer and for any
$i = 1,...,m$ let ${\cal G}_i$  be a real Hilbert space, $B_i:{\cal G }_i \rightrightarrows {\cal G}_i$ a
maximally monotone operator, $D_i:{\cal G}_i\rightrightarrows {\cal G}_i$ a monotone operator such that $D_i^{-1}$ is $\nu_i$-Lipschtz continuous
for $\nu_i>0$ and $L_i:{\cal H}\rightarrow$ ${\cal G}_i$ a nonzero linear continuous operator. Consider also $B:{\cal H}\rightarrow{\cal H}$ a
monotone and $\mu^{-1}$-Lipschitz continuous operator with $\mu>0$ and suppose that $M=\zer B \neq\emptyset$. The monotone inclusion problem to solve is
\begin{equation}\label{sum-k-primal-C-D}
0\in Ax+ \sum_{i=1}^{m}L_i^*(B_i\Box D_i)(L_ix)+Cx+N_M(x).
\end{equation}
\end{problem}

The algorithm we propose for solving this problem has the following form. 

\begin{algorithm}\label{alg-par-sum}$ $

\noindent\begin{tabular}{rl}
\verb"Initialization": & \verb"Choose" $(x_0,v_{1,0},...,v_{m,0}),(x_1,v_{1,1},...,v_{m,1})\in{\cal H} \times$ ${\cal G}_1 \times...\times {\cal G}_m$\\
\verb"For" $n \geq 1$ \verb"set": & $p_n=J_{\lambda_n A}[x_n-\lambda_n(Cx_n+\sum_{i=1}^mL_i^*v_{i,n})-\lambda_n\beta_nBx_n+\alpha_n(x_n-x_{n-1})]$\\
                                  & $q_{i,n}=J_{\lambda_n B_i^{-1}}[v_{i,n}+\lambda_n(L_ix_n-D_i^{-1}v_{i,n})+\alpha_n(v_{i,n}-v_{i,n-1})]$, $i=1,...,m$\\
                                  &  $x_{n+1}=\lambda_n\beta_n(Bx_n-Bp_n)+\lambda_n(Cx_n-Cp_n)$\\
                                  &  \hspace{4.3cm} $+\lambda_n\sum_{i=1}^mL_i^*(v_{i,n}-q_{i,n})+p_n$\\
                                  & $v_{i,n+1}=\!\lambda_n L_i(p_n-x_n)+\!\lambda_n(D_i^{-1}v_{i,n}-D_i^{-1}q_{i,n}) +q_{i,n}, i=1,...,m$,
\end{tabular}
\end{algorithm}
where $(\lambda_n)_{n \geq 1}$, $(\beta_n)_{n \geq 1}$ and $(\alpha_n)_{n\geq 1}$ are sequences of positive real numbers.

\begin{remark} In case $Bx=0$ for all $x\in\cal H$, the above numerical scheme becomes the inertial algorithm that have been studied in \cite{b-c-inertial} in connection with the 
solving of the monotone inclusion problem 
\begin{equation*}
0\in Ax+ \sum_{i=1}^{m}L_i^*(B_i\Box D_i)(L_ix)+Cx.
\end{equation*}

If, additionally, $\alpha_n=0$ for all $n\geq 1$, then the algorithm collapses into the error-free variant of the primal-dual iterative  scheme formulated in \cite[Theorem 3.1]{combettes-pesquet}. 
\end{remark}

For the convergence result we need the following additionally hypotheses (we refer the reader to the remarks \ref{cond} and 
\ref{cond-subdiff} for sufficient conditions guaranteeing $(H_{fitz}^{par-sum})$):

$$(H_{fitz}^{par-sum})\left\{
\begin{array}{lll}
(i) \ A+N_M \mbox{ is maximally monotone and }\\
\zer\big(A+\sum_{i=1}^{m}L_i^*\circ(B_i\Box D_i)\circ L_i +C +N_M\big)\neq\emptyset;\\
(ii) \ \mbox{For every }p\in\ran N_M, \!\sum\limits_{n \geq 1} \!\lambda_n\beta_n\left[\sup\limits_{u\in M}\varphi_B\left(u,\frac{p}{\beta_n}\right)-\sigma_M\left(\frac{p}{\beta_n}\right)\right]<+\infty;\\
(iii) \ (\lambda_n)_{n \geq 1} \in\ell^2\setminus\ell^1.\end{array}\right.$$

For proving the convergence of the sequences generated by Algorithm \ref{alg-par-sum} we will make use of a product space approach, which relies on the reformulation of Problem \ref{pr-par-sum} in 
the same form as Problem \ref{pr-Lip-single-val}.

\begin{theorem}\label{fbf-par-sum} Let be the sequences generated by Algorithm \ref{alg-par-sum}
and let $(z_n)_{n \geq 1}$ the be sequence defined in \eqref{average}. 
Assume that $(H_{fitz}^{par-sum})$ is fulfilled, $(\alpha_n)_{n\geq 1}$ 
is nondecreasing and there exist $n_0\geq 1$, $\alpha\geq 0$ and $\sigma>0$ such that for any $n \geq n_0$\begin{equation}\label{a_n-sum}0\leq\alpha_n\leq \alpha \ \forall n\geq n_0\end{equation} and 
\begin{equation}\label{a-s-sum}5\alpha+2\sigma+(1+4\alpha+2\sigma)\left (\frac{\lambda_n\beta_n}{\mu} + 
\lambda_n\beta \right)^2\leq 1 \ \forall n\geq n_0,\end{equation}
where 
$$\beta=\max\{\nu,\nu_1,...,\nu_m\}+\sqrt{\sum_{i=1}^m\|L_i\|^2}.$$
Then $(z_n)_{n \geq 1}$ converges weakly to an element in $\zer\big(A+\sum_{i=1}^{m}L_i^*\circ(B_i\Box D_i)\circ L_i +C +N_M\big)$ 
as $n\rightarrow+\infty$. If, additionally, $A$ and $B_i^{-1}$, $i=1,...,m,$ are strongly monotone,
then $(x_n)_{n \geq 1}$ converges strongly to the unique element in 
$\zer\big(A+\sum_{i=1}^{m}L_i^*\circ(B_i\Box D_i)\circ L_i +C +N_M\big)$ as $n\rightarrow+\infty$.
\end{theorem}

\begin{proof} The proof makes use of similar techniques as in \cite{bc-viet}, however, for the sake of completeness, we provide as follows the necessary details. 

We start by noticing that $x\in{\cal H}$ is a solution to Problem \ref{pr-par-sum} if and only if there exist
$ v_1 \in {\cal G}_1,...,v_m \in {\cal G}_m$ such that
\begin{equation}\label{sum-k-dual-C-D}
 \
\left\{
\begin{array}{ll}
0\in Ax+\sum_{i=1}^{m}L_i^*v_i+Cx+N_M(x)\\
v_i\in (B_i\Box D_i)(L_ix), i=1,...,m,
\end{array}\right.
\end{equation}

which is nothing else than

\begin{equation}\label{sum-k-dual-C-D2}
 \
\left\{
\begin{array}{ll}
0\in Ax+\sum_{i=1}^{m}L_i^*v_i+Cx+N_M(x)\\
0\in B_i^{-1}v_i+D_i^{-1}v_i-L_ix, i=1,...,m.
\end{array}\right.
\end{equation}

We further endow the product space ${\cal H} \times$ ${\cal G}_1 \times...\times {\cal G}_m$ with inner product and associated norm defined for all
$(x,v_1,...,v_m), (y,w_1,...,w_m)\in {\cal H} \times$ ${\cal G}_1 \times...\times {\cal G}_m$ as
$$\langle (x,v_1,...,v_m),(y,w_1,...,w_m)\rangle=\langle x,y\rangle+\sum_{i=1}^m\langle v_i,w_i\rangle$$ and
$$\|(x,v_1,...,v_m)\|=\sqrt{\|x\|^2+\sum_{i=1}^m\|v_i\|^2},$$
respectively.

We introduce the operators $\widetilde A:{\cal H} \times$ ${\cal G}_1 \times...\times {\cal G}_m\rightrightarrows{\cal H} \times$ ${\cal G}_1 \times...\times {\cal G}_m$
$$\widetilde A(x,v_1,...,v_m)=Ax\times B_1^{-1}v_1\times....\times B_m^{-1}v_m,$$
$\widetilde D: {\cal H} \times$ ${\cal G}_1 \times...\times {\cal G}_m\rightarrow{\cal H} \times$ ${\cal G}_1 \times...\times {\cal G}_m$,
$$\widetilde D(x,v_1,...,v_m)=\Big(\sum_{i=1}^mL_i^*v_i+Cx,D_1^{-1}v_1-L_1x,...,D_m^{-1}v_m-L_mx\Big)$$
and $\widetilde B: {\cal H} \times$ ${\cal G}_1 \times...\times {\cal G}_m\rightarrow{\cal H} \times$ ${\cal G}_1 \times...\times {\cal G}_m$,
$$\widetilde B(x,v_1,...,v_m)=(Bx,0,...,0).$$

Notice that, since $A$ and $B_i$, $i=1,...,m$ are maximally monotone, $\widetilde A$ is maximally monotone, too
(see \cite[Props. 20.22, 20.23]{bauschke-book}). Further, as it was done in \cite[Theorem 3.1]{combettes-pesquet}, one can show that $\widetilde D$  is a
monotone and $\beta$-Lipschitz continuous operator. 

Indeed, let be $(x,v_1,...,v_m),(y,w_1,...,w_m)\in {\cal H} \times$ ${\cal G}_1 \times...\times {\cal G}_m$. By using the monotonicity of $C$ and $D_i^{-1}$, $i=1,...,m$. we have
\begin{align*}
\langle (x,v_1,...,v_m) & -(y,w_1,...,w_m), \widetilde D (x,v_1,...,v_m)- \widetilde D (y,w_1,...,w_m)\rangle\\
& =\langle x-y,Cx-Cy\rangle+\sum_{i=1}^m\langle v_i-w_i,D_i^{-1}v_i-D_i^{-1}w_i\rangle\\
& +\sum_{i=1}^m(\langle x-y,L_i^*(v_i-w_i)\rangle
-\langle v_i-w_i,L_i(x-y)\rangle)\geq 0,
\end{align*}
which shows that $\widetilde D$ is monotone.

The Lipschitz continuity of $\widetilde D$ follows by noticing that
\begin{align*}
& \left\|\widetilde D(x,v_1,...,v_m)-\widetilde D(y,w_1,...,w_m)\right\|\\
\leq & \left\|\left(Cx-Cy,D_1^{-1}v_1-D_1^{-1}w_1,...,D_m^{-1}v_m-D_m^{-1}w_m\right)\right\|\\
& +\left\|\left(\sum_{i=1}^m L_i^*(v_i-w_i),-L_1(x-y),...,-L_m(x-y)\right)\right\|\\
\leq & \sqrt{\nu^2\|x-y\|^2 +\sum_{i=1}^m\nu_i^2\|v_i-w_i\|^2}+\sqrt{\left(\sum_{i=1}^m\|L_i\|\cdot\|v_i-w_i\|\right)^2+\sum_{i=1}^m\|L_i\|^2\cdot\|x-y\|^2}\\
\leq & \beta\|(x,v_1,...,v_m)-(y,w_1,...,w_m)\|.
\end{align*}

Moreover, $\widetilde B$ is monotone, $\mu^{-1}$-Lipschitz continuous and
$$\zer \widetilde B=\zer B\times {\cal G}_1 \times...\times {\cal G}_m=M\times {\cal G}_1 \times...\times {\cal G}_m,$$
hence
$$N_{\widetilde M}(x,v_1,...,v_m)=N_M(x)\times\{0\}\times...\times\{0\},$$
where
$$\widetilde M=M\times {\cal G}_1 \times...\times {\cal G}_m=\zer\widetilde B.$$

Taking into consideration \eqref{sum-k-dual-C-D2}, we obtain that $x\in{\cal H}$ is a solution to Problem \ref{pr-par-sum} if and only if there exist
$v_1 \in {\cal G}_1,...,v_m \in {\cal G}_m$ such that $$(x,v_1,...,v_m)\in\zer (\widetilde A+\widetilde D+N_{\widetilde M}).$$
Conversely, when $(x,v_1,...,v_m)\in\zer (\widetilde A+\widetilde D+N_{\widetilde M})$, then one obviously has $x\in \zer\big(A+\sum_{i=1}^{m}L_i^*\circ(B_i\Box D_i)\circ L_i +C +N_M\big)$.
This means that determining the zeros of $\widetilde A+\widetilde D+N_{\widetilde M}$ will automatically provide a solution to Problem \ref{pr-par-sum}.

Further, notice that 
$$J_{\lambda \widetilde A}(x,v_1,...,v_m)=\left(J_{\lambda A_1}(x),J_{\lambda B_1^{-1}}(v_1),...,J_{\lambda B_m^{-1}}(v_m)\right)$$
for every $(x,v_1,...,v_m)\in{\cal H} \times$ ${\cal G}_1 \times...\times {\cal G}_m$ and every $\lambda > 0$ 
(see \cite[Proposition 23.16]{bauschke-book}). Thus the iterations of Algorithm \ref{alg-par-sum} read for any $n \geq 1$:
$$\left\{
\begin{array}{l}
(p_n,q_{1,n},...,q_{m,n})=J_{\lambda_n \widetilde A}\left[(x_n,v_{1,n},...,v_{m,n})-\lambda_n \widetilde D(x_n,v_{1,n},...,v_{m,n}) \right.\\
\hspace{2cm} \left. - \lambda_n \beta_n\widetilde{B} (x_n,v_{1,n},...,v_{m,n})+\alpha_n\big((x_n,v_{1,n},...,v_{m,n})-(x_{n-1},v_{1,n-1},...,v_{m,n-1})\big)\right]\\
(x_{n+1},v_{1,n+1},...,v_{m,n+1})=\lambda_n\beta_n\left[\widetilde{B}(x_n,v_{1,n},...,v_{m,n})-\widetilde{B}(p_n,q_{1,n},...,q_{m,n})\right]\\
\hspace{2cm} +\lambda_n\left[\widetilde D(x_n,v_{1,n},...,v_{m,n})-\widetilde D(p_n,q_{1,n},...,q_{m,n})\right]+(p_n,q_{1,n},...,q_{m,n}),
\end{array}\right.$$
which is nothing else than the iterative scheme of Algorithm \ref{alg-fbf} employed to the solving of the monotone inclusion problem
$$0 \in \widetilde Ax+\widetilde Dx+N_{\widetilde M}(x).$$

In order to compute the Fitzpatrick function of $\widetilde B$, we consider two arbitrary elements 
$(x,v_1,...,v_m)$, $(x',v_1',...,v_m') \in {\cal H} \times$ ${\cal G}_1 \times...\times {\cal G}_m$. It holds
\begin{align*}
& \varphi_{\widetilde B}\left((x,v_1,...,v_m),(x',v_1',...,v_m')\right)=\\
& \sup_{\substack{(y,w_1,...,w_m) \in \\ {\cal H} \times {\cal G}_1 \times...\times {\cal G}_m}}\Big\{\langle (x,v_1,...,v_m),\widetilde B(y,w_1,...,w_m)\rangle+\langle (x',v_1',...,v_m'),(y,w_1,...,w_m)\rangle\\
& \qquad \qquad \qquad -\langle (y,w_1,...,w_m),\widetilde B(y,w_1,...,w_m)\rangle\Big\}\\
& =\sup_{\substack{(y,w_1,...,w_m) \in \\ {\cal H} \times {\cal G}_1 \times...\times {\cal G}_m}} \left\{\langle x, By\rangle+\langle x', y\rangle+\sum_{i=1}^m\langle v_i', w_i\rangle-\langle y, By\rangle\right\},
\end{align*}
 thus
$$\varphi_{\widetilde{B}}\big((x,v_1,...,v_m),(x',v_1',...,v_m')\big)=\left\{
\begin{array}{ll}
\varphi_B(x,x'), & \mbox {if } v_1'=...=v_m'=0,\\
+\infty, & \mbox{otherwise.}
\end{array}\right.$$

Moreover, $$\sigma_{\widetilde{M}}(x,v_1,...,v_m)=\left\{
\begin{array}{ll}
\sigma_M(x), & \mbox {if } v_1=...=v_m=0,\\
+\infty, & \mbox{otherwise,}
\end{array}\right.$$
hence condition (ii) in $(H_{fitz}^{par-sum})$ is nothing else than
\begin{align*}
& \mbox{for each} \ (p,p_1,...,p_m)\in\ran N_{\widetilde M}=\ran N_M\times\{0\}\times...\times\{0\}, \\
& \sum_{n \geq 1} \lambda_n\beta_n\!\left[\sup\limits_{(u,v_1,...,v_m)\in \widetilde M}\varphi_{\widetilde B}\left((u,v_1,...,v_m),\frac{(p,p_1,...,p_m)}{\beta_n}\right)
-\sigma_{\widetilde M}\left(\frac{(p,p_1,...,p_m)}{\beta_n}\right)\right]\!<\!+\infty.
\end{align*}
Moreover, condition (i) in $(H_{fitz}^{par-sum})$ ensures that $\widetilde A+N_{\widetilde M}$ is
maximally monotone and $\zer(\widetilde A+\widetilde D+N_{\widetilde M})\neq\emptyset$. Hence, we are in the position of applying
Theorem \ref{fbf-conv} in the context of finding the zeros of $\widetilde A+\widetilde D+N_{\widetilde M}$. The statements of the theorem are an easy consequence of this result.
\end{proof}

\section{Convex minimization problems}\label{sec4}

In this section we deal with the minimization of a
complexly structured convex objective function subject to the set of minima of another convex and differentiable function with Lipschitz continuous gradient. We show how the 
results obtained in the previous section for monotone inclusion problems can be applied in this context. 

\begin{problem}\label{pr-opt}
Let ${\cal H}$ be a real Hilbert space, $f\in\Gamma({\cal H})$ and $h:{\cal H}\rightarrow \R$ be a convex and differentiable
function with a $\nu$-Lipschitz continuous gradient for $\nu>0$. Let $m$ be a strictly positive integer and for any $i = 1,...,m$ let ${\cal G}_i$
be a real Hilbert space, $g_i, l_i \in\Gamma({\cal G}_i)$ such that $l_i$ is $\nu_i^{-1}$-strongly convex for $\nu_i > 0$ and
$L_i:{\cal H}\rightarrow$ ${\cal G}_i$ a nonzero linear continuous operator. Further, let  $\Psi\in\Gamma (\cal H)$ be differentiable with a $\mu^{-1}$-Lipschitz continuous gradient, fulfilling $\min \Psi=0$.
The convex minimization problem under investigation is
\begin{equation}\label{opt}
\inf_{x\in \argmin \Psi}\left\{f(x)+\sum_{i=1}^{m}(g_i \Box l_i)(L_ix)+h(x)\right\}.
\end{equation}
\end{problem}

Consider the maximal monotone operators
$$A=\partial f, B= \nabla\Psi, C=\nabla h, B_i=\partial g_i \ \mbox{and} \ D_i=\partial l_i, i=1,...,m.$$
According to \cite[Proposition 17.10, Theorem 18.15]{bauschke-book}, $D_i^{-1} = \nabla l_i^*$ is a monotone
and $\nu_i$-Lipschitz continuous operator for $i=1,...,m$. Moreover, $B$ is a monotone and $\mu^{-1}$-Lipschitz continuous operator and
$$M:=\argmin \Psi = \zer B.$$

Taking into account the sum rules of the convex subdifferential, every element of
$\zer\big(\partial f+\sum_{i=1}^{m}L_i^*\circ(\partial g_i\Box \partial l_i)\circ L_i +\nabla h +N_M\big)$ is an optimal solution of \eqref{opt}. The converse is true if an appropriate qualification condition is satisfied.
For the readers convenience, we present the following qualification condition
of interiority-type (see, for instance, \cite[Proposition 4.3, Remark 4.4]{combettes-pesquet})
\begin{equation}\label{reg-cond} (0,...,0)\in\sqri\left(\prod_{i=1}^{m}(\dom g_i+\dom l_i)-\{(L_1x,...,L_mx):x\in \dom f\cap M\}\right).
\end{equation}
The condition \eqref{reg-cond} is fulfilled in one of the following circumstances:

(i) $\dom g_i+\dom l_i={\cal{G}}_i$, $i=1,...,m$; 

(ii) ${\cal H}$ and
${\cal{G}}_i$ are finite-dimensional and there exists $x\in\ri\dom f\cap\ri M$ such that $L_ix\in\ri\dom g_i+\ri\dom l_i$, $i=1,...,m$
(see \cite[Proposition 4.3]{combettes-pesquet}).

Algorithm \ref{alg-par-sum} becomes in this particular case

\begin{algorithm}\label{alg-opt}$ $

\noindent\begin{tabular}{rl}
\verb"Initialization": & \verb"Choose" $(x_0,v_{1,0},...,v_{m,0}),(x_1,v_{1,1},...,v_{m,1})\in{\cal H} \times$ ${\cal G}_1 \times...\times {\cal G}_m$\\
\verb"For" $n \geq 1$ \verb"set": & $p_n=\prox _{\lambda_n f}[x_n-\lambda_n(\nabla h(x_n)+\sum_{i=1}^mL_i^*v_{i,n})-\lambda_n\beta_n\nabla\Psi (x_n)$\\
&$\hspace{2.1cm} +\alpha_n(x_n-x_{n-1})]$\\
                                  & $q_{i,n}\!=\prox _{\lambda_n g_i^*}\![v_{i,n}+\!\lambda_n(L_ix_n-\nabla l_i^*(v_{i,n}))\!+\alpha_n(v_{i,n}-v_{i,n-1})], \!i=1,...,m$\\
                                  &  $x_{n+1}=\lambda_n\beta_n(\nabla\Psi (x_n)-\nabla\Psi (p_n))+\lambda_n(\nabla h(x_n)-\nabla h(p_n))$\\
                                  & $\hspace{1.2cm}+\lambda_n\sum_{i=1}^mL_i^*(v_{i,n}-q_{i,n})+p_n$\\
                                  & $v_{i,n+1}=\!\lambda_n L_i(p_n-x_n)+\!\lambda_n(\!\nabla l_i^*(v_{i,n})-\!\nabla l_i^*(q_{i,n})) +\!q_{i,n}, i=1,...,m.$
\end{tabular}
\end{algorithm}

For the convergence result we need the following hypotheses:
$$(H_{fitz}^{opt})\left\{
\begin{array}{lll}
(i) \ \partial f+N_M \mbox{ is maximally monotone and} \ \eqref{opt} \ \mbox{has an optimal solution}; \\
(ii) \ \mbox{ For every }p\in\ran N_M, \sum_{n \geq 0} \lambda_n\beta_n\left[\Psi^*\left(\frac{p}{\beta_n}\right)-\sigma_M\left(\frac{p}{\beta_n}\right)\right]<+\infty;\\
(iii) \ (\lambda_n)_{n \geq 1} \in\ell^2\setminus\ell^1.\end{array}\right.$$

\begin{remark}\label{cond-subdiff} (a) Let us mention that $\partial f+N_M$ is maximally monotone, if $0\in \sqri(\dom f-M)$, a condition which is
fulfilled if, for instance, $f$ is continuous at a point in $\dom f\cap M$ or $\inte M\cap\dom f\neq\emptyset$.

(b) Since $\Psi(x)=0$ for all $x\in M$, by \eqref{fitzp-subdiff-ineq} it follows that whenever (ii) in
$(H_{fitz}^{opt})$ holds, condition (ii) in $(H_{fitz}^{par-sum})$, formulated for $B= \nabla\Psi$, is also true.

(c) The hypothesis (ii) is satisfied, if $\sum_{n \geq 1} \frac{\lambda_n}{\beta_n}<+\infty$ and $\Psi$ is bounded
below by a multiple of the square of the distance to $C$ (see \cite{att-cza-peyp-p}). This is for instance the case when
$M=\zer L = \{x \in {\cal H}: Lx=0\}$, $L : {\cal H} \rightarrow {\cal H}$ is a linear continuous operator with closed range and
$\Psi : {\cal H} \rightarrow \R, \Psi(x)=\|Lx\|^2$ (see \cite{att-cza-peyp-p, att-cza-peyp-c}). For further situations for which condition
(ii) is fulfilled we refer to \cite[Section 4.1]{att-cza-peyp-c}.
\end{remark}

We are able now to formulate the convergence result.

\begin{theorem}\label{fbf-opt} Let be the sequences generated by Algorithm \ref{alg-opt}
and let $(z_n)_{n \geq 1}$ the be sequence defined in \eqref{average}. 
Assume that $(H_{fitz}^{opt})$ is fulfilled, $(\alpha_n)_{n\geq 1}$ 
is nondecreasing and there exist $n_0\geq 1$, $\alpha\geq 0$ and $\sigma>0$ such that for any $n \geq n_0$\begin{equation}\label{a_n-sum-opt}0\leq\alpha_n\leq \alpha \ \forall n\geq n_0\end{equation} and 
\begin{equation}\label{a-s-sum-opt}5\alpha+2\sigma+(1+4\alpha+2\sigma)\left (\frac{\lambda_n\beta_n}{\mu} + 
\lambda_n\beta \right)^2\leq 1 \ \forall n\geq n_0,\end{equation}
where 
$$\beta=\max\{\nu,\nu_1,...,\nu_m\}+\sqrt{\sum_{i=1}^m\|L_i\|^2}.$$
Then $(z_n)_{n \geq 1}$ converges weakly to an optimal solution to \eqref{opt} as $n\rightarrow+\infty$.
If, additionally, $f$ and $g_i^*$, $i=1,...,m$, are strongly convex,
then $(x_n)_{n \geq 1}$ converges strongly to the unique optimal solution of \eqref{opt} as $n\rightarrow+\infty$.
\end{theorem}

\begin{remark} (a) According to \cite[Proposition 17.10, Theorem 18.15]{bauschke-book}, for a function $g \in \Gamma({\cal H})$ one has that $g^*$ is strongly convex if and only if $g$ is differentiable with Lipschitz continuous gradient.

(b) Notice that in case $\Psi (x)=0$ for all $x\in\cal H$ Algorithm \ref{alg-opt} has been studied in 
\cite{b-c-inertial} in connection with the solving of the optimization problem \begin{equation}\label{opt2}
\inf_{x\in \cal H}\left\{f(x)+\sum_{i=1}^{m}(g_i \Box l_i)(L_ix)+h(x)\right\}.
\end{equation}
If, additionally, $\alpha_n=0$ for all $n\geq 1$, then Algorithm \ref{alg-opt} becomes the error-free variant of the iterative scheme 
given in \cite[Theorem 4.2]{combettes-pesquet} for solving \eqref{opt2}. 
\end{remark}

\end{document}